\documentclass{ifacconf}

\usepackage{natbib}            
\usepackage{graphicx}          

\usepackage{units}
\usepackage{amsmath}
\usepackage{amssymb}
\usepackage[T1]{fontenc}
\usepackage{lmodern}

\newcommand{\degree}{\ensuremath{^\circ}}
\newtheorem{proof}[thm]{Proof: }

\makeatother

\begin{document}

\begin{frontmatter}

\title{Trajectory control of a bipedal walking robot with inertial disc\thanksref{footnoteinfo}} 

\thanks[footnoteinfo]{Carlos Novaes thanks CNPq for the financial support.}

\author[First]{Carlos Eduardo de Brito Novaes}
\author[Second]{Paulo Sergio Pereira da Silva}
\author[Third]{Pierre Rouchon}

\address[First]{Escola Polit\'{e}cnica da USP (e-mail: carlos.novaes@usp.br)}
\address[Second]{Escola Polit\'{e}cnica da USP (e-mail: paulo@lac.usp.br)}
\address[Third]{Mines ParisTech  (e-mail: pierre.rouchon@mines-paristech.fr)}

\begin{keyword}                           
Nonlinear control, Walking, Robots, Hybrid model, Autonomous mobile robots
\end{keyword}                             

\begin{abstract}
In this paper we exploit some interesting properties of a class of
bipedal robots which have an inertial disc. One of this properties
is the ability to control every position and speed except for the disc position.
The proposed control is designed in two hierarchic levels. The first will drive the robot geometry, while
the second will control the speed and also the angular momentum. The exponential stability of this approach
is proved around some neighborhood of the nominal trajectory defining the geometry of the step. This control
will not spend energy to adjust the disc position and neither to synchronize
the trajectory with the time. The proposed control only takes action
to correct the essential aspects of the walking gait.
Computational simulations are presented for different conditions, serving as a empirical
test for the neighborhood of attraction.\end{abstract}

\end{frontmatter}

\section{Introduction}

Dynamic robot locomotion is a particularly challenging study. One
of its main difficulties is due to the fact that the robot is sub-actuated,
that is, on a dynamic walking gait it is impossible to freely control each orientation of the robot's links. It is straightforward
to see this limitation as there is no actuator between the ground and the support polygon,
the gravity alone will impose an additional torque that will change the robot orientation. Aside from this, a dynamic walking
bipedal robot can achieve greater speeds than a static walker, so
there is a practical interest in this domain of research.

In \cite{GrizzleEtAl_FeedbackControlDynamicBipedal} the problem of
sub-actuated dynamic walking is treated in a systematic way that will
lead to controller design with assured stability and also, with a
step evolution clocked, not by the time, but by its own geometry.
This self clocked characteristic will be instrumental in this work.
We can see each step as a trajectory for the state vector of the dynamic
model. The classical controller design will ensure the state of the
model to track some reference tied to the time and some control effort
will take place even if the robot state is valid for the reference
trajectory but is not synchronized with the reference.

To successfully perform a step, the robot links must be driven in
some specific way, even for coordinates that cannot be directly controlled.
To accomplish this it may be necessary, for example, to swing the
robot's torso and this can be undesirable in some cases.

To have complete control over one coordinate, it is necessary to give
up the control of another one. So, if it is possible to ignore the position
or the orientation of a link, it will be possible to control every
other important coordinate. By using a inertial disc, its orientation
will be a cyclic variable, and thus, could be ignored. That is the
main idea introduced by the design presented in \cite{Kieffer1993_Biped_Walking}
and studied in many works as \cite{Rouchon_Ramirez_ControlOfTheWalkingToy},
\cite{AlmostLinearBiped_MarkWSpong} and \cite{Peres2008_Robo_Bipede}.
Theoretically it is possible to ignore the disc position and speed,
but in practice, there will be some limitation on the maximum speed
achieved by the disc, so that the actuators do not saturate. One solution
to this problem is proposed in \cite{Peres2008_Robo_Bipede} by means
of a supervisory control, that will perform a different trajectory
if the disc speed is beyond some limit value.

Section \ref{sec:Dynamic-Model} presents the hybrid dynamic model
as proposed in \cite{GrizzleEtAl_FeedbackControlDynamicBipedal}.
There is also a brief presentation of the hypotheses and terminology
therein and inherited by this paper.

Next, it will be introduced the trajectory planning in section \ref{sec:Design-of-theTrajectory}.
This is the main practical benefit of our approach, valid only for
this specific class of bipedal robots. In \cite{GrizzleEtAl_FeedbackControlDynamicBipedal},
the proposed strategy is to find the evolution of the robot geometry
as function of the absolute orientation and then check if it will
lead to a monotonic evolution of that absolute orientation. In our
approach, it is possible to fix the robot geometry as a desired function
of the absolute orientation. Then find a monotonic evolution of the
absolute orientation that lead to a repeatable evolution of the disc
speed. Another benefit is that the proposed control does not deal
at all with the disc position, but can continuously control its speed.
The control law and the convenient change of coordinates is also presented,
introducing the core contribution of this paper.

Section \ref{sec:ControlStability} present a demonstration of exponential
stability around some neighborhood of the planed trajectory.

Finally in section \ref{sec:Numerical-simulations} there are some
simulation results. \nocite{RobotDynamicsAndControl_SpongVidyasagar}

\section{Dynamic Model}\label{sec:Dynamic-Model}
The robot will be modeled by a chain of $N$ rigid links, each one
with known parameters such as mass, center of mass and inertia. The
links are connected by $N-1$ frictionless joints independently actuated.
There will be at least one and at most two point of contact with the
ground, called foots. The foots are punctual and unactuated.

For illustration purposes, a very simple bipedal robot is depicted
at figure \ref{fig:Chap2_ExampleBiped}. The coordinate $q_{N}$ is
the only absolute coordinate, $q_{d}$ is a relative coordinate to
represent the disc position and $\mathbf{q}_{r}$ is a relative coordinate
representing the angular displacement between the two legs.
\begin{figure}
\begin{centering}
\includegraphics[width=0.8\columnwidth]{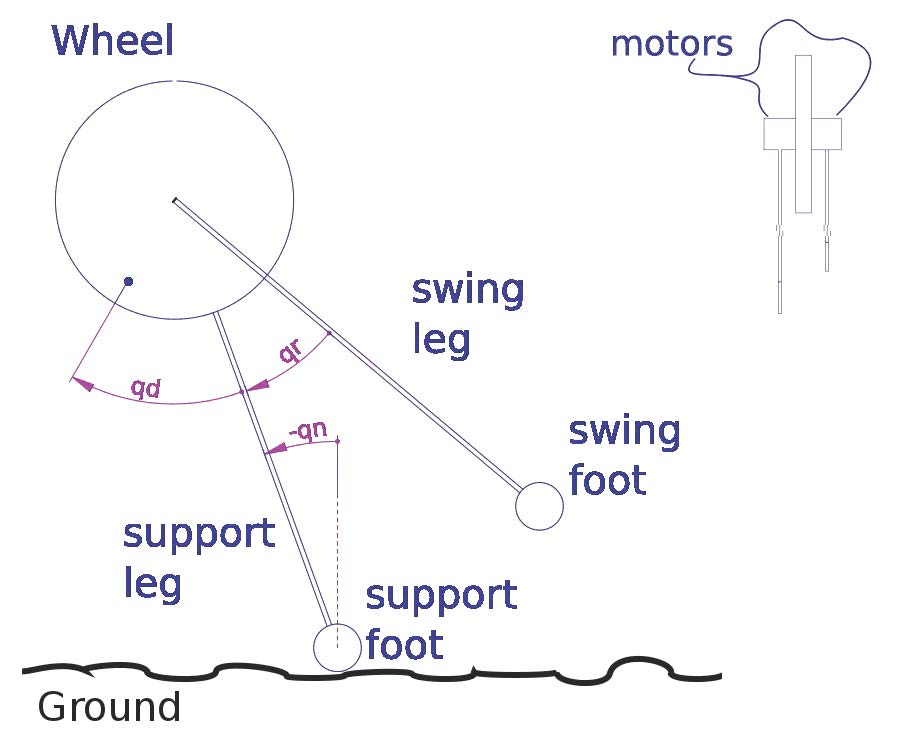}
\par\end{centering}

\caption{Example model. Side view with coordinates and detail of frontal view
in the upper right corner.\label{fig:Chap2_ExampleBiped}}
\end{figure}
 On a higher order robots $\mathbf{q}_{r}$ will be a vector representing
the angular displacements between each link. Indeed, the robot's shape
will be determined by $\mathbf{q}_{r}$.

The walking takes place on a surface, is restricted to the sagittal
plane, and is composed of alternating phases of single and double
support. During the single support phase, the stance foot acts as
an ideal pivot. The double support phase is instantaneous and associated
impacts are modeled as a rigid contact. At the impact, the swing leg
touches the ground with no slip nor rebound and, the former stance
leg releases without interaction.

The robot model is derived under the hypotheses briefly cited above
and detailed in \cite{GrizzleEtAl_FeedbackControlDynamicBipedal},
HR1\textasciitilde{}HR5, HGW1\textasciitilde{}HGW7 and HI1\textasciitilde{}HI7
(\cite[p48~p50]{GrizzleEtAl_FeedbackControlDynamicBipedal}). A small
change is introduced in HR6.
\begin{itemize}
\item [{$\textrm{HR\ensuremath{6^{*}}}$)}] the model is expressed in one
relative coordinate $q_{d}$ for the disc, N-2 relative coordinates
$\mathbf{q}_{r}$ for the rest of the body and only one absolute coordinate
$q_{N}$.
\end{itemize}
Using Lagrangian formulation, the dynamic
equations for the continuous phase, between two impacts, will be

{\small
\begin{eqnarray}
\mathbf{D}\left(\mathbf{q}_{r}\right)\mathbf{\ddot{q}}+\mathbf{C}\left(\mathbf{q}_{r},\mathbf{\dot{q}}_{r},\dot{q}_{N}\right)\dot{\mathbf{q}}+\mathbf{G}\left(\mathbf{q}_{r},q_{N}\right) & = & \mathbf{B}\left(\mathbf{q}_{r}\right)\mathbf{u}\label{eq:Chap2_DynamicModel}
\end{eqnarray}
}where{\small
\begin{eqnarray}
\mathbf{q} & \triangleq & \left(\begin{array}{ccc}
q_{d} & \mathbf{q}_{r} & q_{N}\end{array}\right)^{T}\label{eq:CoordinateVector_qs}
\end{eqnarray}
}{\small \par}

$q_{d}\in\mathbb{R}$, $\mathbf{q}_{r}\in\mathbb{R}^{N-2}$, $q_{N}\in\mathbb{R}$
and the vector of actuator torques being defined by $\mathbf{u}\in\mathbb{R}^{N-1}$.
$\mathbf{D}\left(\mathbf{q}_{r}\right)$ is called the Inertia Matrix,
$\mathbf{C}\left(\mathbf{q}_{r},\dot{\mathbf{q}}_{r}\right)$ is called
the Coriolis Matrix, $\mathbf{G}\left(\mathbf{q}_{r},q_{N}\right)$
is called the Gravity Vector and $\mathbf{B}\left(\mathbf{q}_{r}\right)$
will map the actuator torques as generalized forces.

It will be more convenient to define{\small
\begin{eqnarray}
\omega & \triangleq & \left(\begin{array}{ccc}
\sigma_{N} & \dot{\mathbf{q}}_{r} & \dot{q}_{N}\end{array}\right)^{T}\label{eq:coordinate_omega}
\end{eqnarray}
}where $\sigma_{N}$ is the angular moment of the system, calculated
as {\small
\begin{eqnarray}
\sigma_{N} & = & \mathbf{D}_{N}\left(\mathbf{q}_{r}\right)\dot{\mathbf{q}}\label{eq:Chap2_Momentum}
\end{eqnarray}
} $\mathbf{D}_{N}$ being the last line of the inertia matrix. Its
derivative will be

{\small
\begin{eqnarray}
\dot{\sigma}_{N} & = & -\mathbf{G}_{N}\left(\mathbf{q}_{r},q_{N}\right)\label{eq:Chap2_MomentumDerivative}
\end{eqnarray}
}$\mathbf{G}_{N}$ being the last line of the gravity vector.

There will be a diffeomorphism between $\dot{\mathbf{q}}$ and $\omega$
and so it is possible to define the state vector{\small
\begin{eqnarray}
\mathbf{x} & \triangleq & \left(\begin{array}{cc}
\mathbf{q} & \omega\end{array}\right)^{T}\label{eq:vector_state}
\end{eqnarray}
}{\small \par}

and denote by $\mathcal{X}\subset\mathbb{R}^{N}$ the set of valid
values for the vector state $\mathbf{x}$.

At some moment, the swing foot will touch the ground with state $\mathbf{x}^{-}$
and the system will be mapped to a new state $\mathbf{x}^{+}$. It
is assumed that the walking gait is transversal to the impact surface.
So, the impact forces will change the speeds on $\omega$. After the
impact the roles of the legs are reversed with a change in $\mathbf{q}$.
As result, the system state, $\mathbf{x}^{+}$, after the impact is

{\small
\begin{eqnarray}
\mathbf{x}^{+} & = & \Delta\left(\mathbf{x}^{-}\right)\label{eq:ImpactMap_omega}\\
\Delta & = & \left(\begin{array}{c}
\Delta_{\mathbf{q}}\mathbf{q}_{s}^{-}\\
\Delta_{\omega}\left(\mathbf{q}_{r}^{-},q_{N}^{-}\right)\omega^{-}
\end{array}\right)
\end{eqnarray}
}with $\Delta_{\mathbf{q}}$ being a constant involutive matrix and
$\Delta_{\omega}$ a function of the robot geometry $\mathbf{q}_{r}^{-}$
and absolute orientation $q_{N}^{-}$. Note that any value with a
superscript $\mathbf{q}^{-}$ means the value just before the impact
at the end of the step, and a superscript $\mathbf{q}^{+}$ means
the value immediately after the impact at the beginning of a new step.
This is the same notation utilized in \cite{GrizzleEtAl_FeedbackControlDynamicBipedal}.

The impact, modeled by the mapping $\Delta\left(\mathbf{x}^{-}\right)$,
will take place at a given configuration that can be represented by
a manifold $\mathcal{S}$, for example{\small
\begin{eqnarray}
\mathcal{S} & \triangleq & \left\{ \mathbf{x}\in\mathcal{X}|p_{2}\left(\mathbf{q}\right)=0,\dot{p}_{2}\left(\mathbf{x}\right)<0\right\} \label{eq:ImpactManifold}
\end{eqnarray}
}where $p_{2}\left(\mathbf{q}\right)\geq0$ is a convenient%
\footnote{The function $p_{2}\left(\mathbf{x}\right)$ can express for example,
the vertical position of the swing foot and will reach zero only when
the swing foot touch the ground. 
} function for which $\dot{p}_{2}\left(\mathbf{x}^{+}\right)>0$ and
$\dot{p}_{2}\left(\mathbf{x}^{-}\right)<0$.
The complete hybrid model will be

{\small
\begin{equation}
\begin{cases}
\dot{\mathbf{x}}=\mathbf{f}\left(\mathbf{x}\right)+\mathbf{g}\left(\mathbf{x}\right)\mathbf{u} & \mathbf{x}\notin\mathcal{S}\\
\mathbf{x}^{+}=\Delta\left(\mathbf{x}^{-}\right) & \mathbf{x}\in\mathcal{S}
\end{cases}\label{eq:HybridSystem}
\end{equation}
}{\small \par}

For this particular class of robots, if the coordinate vector is choose
as (\ref{eq:vector_state}) then it is possible to demonstrate that
the first column of $\Delta_{\omega}$ is always $\left(\begin{array}{cccc}
1 & 0 & \cdots & 0\end{array}\right)^{T}$. This is due to $q_{d}$ being a cyclic coordinate.

\section{Design of the trajectory}\label{sec:Design-of-theTrajectory}

In this section we present some aspects of the trajectory generation.

\subsection{Ensuring a geometry tied to the absolute orientation}

The idea of a geometry tied to the absolute orientation is presented
in \cite{GrizzleEtAl_FeedbackControlDynamicBipedal}. In this kind
of design, all the actuators will be used to drive the robot geometry
as a function of the absolute orientation. This geometry must be chosen
in a way that, apart from not being directly controlled, the absolute
orientation will have a monotonic evolution. Also, the interaction
of the robot with the ground must produce a stable limit cycle.

In our robot, the disc position will be ignored. This allows us to
get a new input available. Every relative link orientation, except
for the disc coordinate $q_{d}$, should track a predefined function
of $q_{N}$. This way, at any time, the robot configuration will be
a function of its absolute orientation. This kind of control is equivalent
to drive $\mathbf{q}_{r}-\mathbf{h}_{r}^{ref}\left(q_{N}\right)$
to zero on the output defined as

{\small
\begin{eqnarray}
\mathbf{h} & = & \left(\begin{array}{cc}
\mathbf{h}_{r} & q_{N}\end{array}\right)^{T}\label{eq:Chap2_OurOutput}\\
\mathbf{h}_{r} & = & \mathbf{q}_{r}-\mathbf{h}_{r}^{ref}\left(q_{N}\right)
\end{eqnarray}
}with $\mathbf{h}_{r}^{ref}\left(q_{N}\right)$ being the desired
evolution of $\mathbf{q}_{r}$ as a function of $\mathbf{q}_{N}$.
The output described by (\ref{eq:Chap2_OurOutput}) has dimension
$N-1$ which is the same dimension of the input $\mathbf{u}$.
The dynamic system can be written in terms of the global diffeomorphism defined by{\small
\begin{eqnarray}
\mathbf{p} & = & \left(\mathbf{h},q_{d},\dot{\mathbf{h}},\sigma_{N}\right)\label{eq:Coordinate_p}
\end{eqnarray}}

If the decoupling matrix $\mathcal{L}_{g}\mathcal{L}_{f}\mathbf{h}$
is square%
\footnote{It will be square as $\mathbf{h}$ and $\mathbf{u}$ have both the
same dimension.%
} and invertible, there exists a input
{\small
\begin{eqnarray}
\mathbf{u}^{*} & = & \left(\mathcal{L}_{g}\mathcal{L}_{f}\mathbf{h}\right)^{-1}\left\{ \mathbf{v}-\mathcal{L}_{f}^{2}\mathbf{h}\right\} \label{eq:Chap2_LinearInput}
\end{eqnarray}
}so that by the use of $\mathbf{u}^{*}$ and the coordinate change (\ref{eq:Coordinate_p}),
the system can be viewed as $N-1$ independent series of 2 integrators driven
by the virtual input $\mathbf{v}=\left(\mathbf{v}_{r},v_{N}\right)^{T}$.

Thus, if it is possible to find $\mathbf{u}^{*}$, it will be possible
to control every entry of $\mathbf{h}$ by the input $\mathbf{v}$.
More information on the exact linearization control can be found in
\cite{Isidori:NonLinearControlSystems_3ed}.

Applying the change of coordinates defined by (\ref{eq:Coordinate_p})
it is possible to write {\small
\begin{eqnarray}
\dot{\sigma}_{N} & = & k_{2}\left(\mathbf{h}_{r},q_{N}\right)\label{eq:MomentumDerivative_Restricted}\\
k_{2}\left(\mathbf{h}_{r},q_{N}\right) & \triangleq & -\mathbf{G}_{N}\left(\mathbf{h}_{r}+\mathbf{h}_{r}^{ref}\left(q_{N}\right),q_{N}\right)
\end{eqnarray}
}{\small \par}

Using the input (\ref{eq:Chap2_LinearInput}), the change of coordinates
(\ref{eq:Coordinate_p}) and discarding $q_{d}$, it is possible to
rewrite the continuous phase as{\small
\begin{eqnarray}
\cfrac{d}{dt}\left(\begin{array}{c}
h_{r}\\
\dot{h}_{r}\\
\sigma_{N}\\
q_{N}\\
\dot{q}_{N}
\end{array}\right) & = & \left(\begin{array}{c}
\dot{h}_{r}\\
0\\
k_{2}\left(\mathbf{h}_{r},q_{N}\right)\\
\dot{q}_{N}\\
0
\end{array}\right)+\left(\begin{array}{c}
0\\
\mathbf{v}_{r}\\
0\\
0\\
v_{N}
\end{array}\right)\label{eq:LinearizedSystem}
\end{eqnarray}
}{\small \par}

Now define an embedded manifold

{\small
\begin{eqnarray}
\mathcal{W} & \triangleq & \left\{ \mathbf{x}\in\mathcal{X}|h_{r}=0,\,\dot{h}_{r}=0\right\}
\end{eqnarray}
}{\small \par}

From equation (\ref{eq:LinearizedSystem}) it is possible to see that
the input $\mathbf{v}_{r}$ can be used to drive $\mathbf{h}_{r}=0$.
The input $v_{N}$ remains free to control the absolute speed
$\dot{q}_{N}$ and also, when restricted to $\mathcal{W}$, the angular momentum $\sigma_{N}$.

Assume that it is possible to find functions $V\left(q_{N}\right)>0$ and $S\left(q_{N}\right)$, such that if $\dot{q}_{N}=V\left(q_{N}\right)$, then $\sigma_{N}=S\left(q_{N}\right)+\mathbb{C}_{\sigma}$, where
$\mathbb{C}_{\sigma}$ is some constant bias. It is possible to use
the input $v_{N}$ to exponentially attenuate this bias.

\subsection{Finding a periodic step}
The manifold $\mathcal{W}$ is said to be forward invariant if solutions
starting at $\mathcal{W}$ will remain in $\mathcal{W}$. It will
be said to be impact invariant if $\mathcal{W}\cap\mathcal{S}\neq\emptyset$
and $\Delta\left(\mathcal{W}\cap\mathcal{S}\right)\subset\mathcal{W}$.
Observe that $\Delta\left(\mathcal{W}\cap\mathcal{S}\right)\cap\mathcal{S}=\emptyset$.
If $\mathcal{W}$ is both forward and impact invariant, it is said
to be hybrid invariant \cite[p 96]{GrizzleEtAl_FeedbackControlDynamicBipedal}.
While the feedback control can lead to forward invariance, the impact
invariance is a design property.

For the nominal trajectory to be repeated, each variable evolution
must be hybrid invariant. So, not only $\mathcal{W}$, but also the
angular momentum reference $S\left(q_{N}\right)$, and the absolute
speed reference $V\left(q_{N}\right)$, has to be hybrid invariant
themselves.

One procedure to find such nominal trajectory would be:
\begin{enumerate}
\item Fix the absolute orientations at the beginning, $q_{N}=\bar{q}_{N}^{+}$,
and at the end, $q_{N}=\bar{q}_{N}^{-}$, of the step, and the nominal
configurations $\bar{\mathbf{q}}^{+}$ and $\bar{\mathbf{q}}^{-}$.
\item Determine a desirable evolution $\mathbf{h}_{r}^{ref}\left(q_{N}\right)$
for the robot configuration, such that

\begin{enumerate}
\item {\small $\bar{\mathbf{q}}^{+}=\Delta_{\mathbf{q}}\bar{\mathbf{q}}^{-}\forall\mathbf{q}\in\mathcal{W}$},
as $q_{d}$ is a cyclic coordinate, it can be ignored; this will ensure
impact invariance on the robot shape $\mathbf{h}_{r}^{ref}\left(q_{N}\right)$
\item depending on the value of $\Delta_{\omega}\left(\bar{\mathbf{q}}_{r}^{-},\bar{q}_{N}^{-}\right)$,
determine $\cfrac{\partial\mathbf{h}_{r}^{ref}\left(q_{N}^{+}\right)}{\partial q_{N}}$
and $\cfrac{\partial\mathbf{h}_{r}^{ref}\left(q_{N}^{-}\right)}{\partial q_{N}}$
for impact invariance of $\dot{\mathbf{h}}_{r}^{ref}\left(q_{N}\right)$.
There will be also a fixed increment or decrement of the angular momentum
$\Delta_{\sigma}=\sigma_{N}^{-}-\sigma_{N}^{+}$. This will lead to
impact invariance on $\mathcal{W}$.
\end{enumerate}
\item choose a impact invariant and convenient function candidate for $V\left(q_{N}\right)>0$
, then numerically calculate {\small
\begin{eqnarray}
S\left(q_{N}\right) & = & \intop_{q_{N}^{+}}^{q_{N}}\tfrac{k_{2}\left(0,\tau\right)}{V\left(\tau\right)}d\tau+\sigma_{N}^{+}
\end{eqnarray}
}such that $\Delta_{\sigma}=S\left(q_{N}^{-}\right)-S\left(q_{N}^{+}\right)$
will satisfy the impact invariance restriction for $S\left(q_{N}\right)$.
\end{enumerate}
The control must ensure forward invariance and also a stability around this planned
trajectory.

\subsection{A new change of coordinates}
By construction of $V\left(q_{N}\right)$ and $S\left(q_{N}\right)$,
when restricted to $\mathcal{W}$, where $\mathbf{h}_{r}=\mathbf{0}$,
the following relation holds{\small
\begin{eqnarray}
\left.k_{2}\left(q_{N},\mathbf{h}_{r}\right)\right|_{\mathcal{W}} & = & \cfrac{\partial S\left(q_{N}\right)}{\partial q_{N}}V\left(q_{N}\right)
\end{eqnarray}
}and so it is possible to write{\small
\begin{eqnarray}
k_{2}\left(\mathbf{h}_{r},q_{N}\right)-\cfrac{\partial S\left(q_{N}\right)}{\partial q_{N}}V\left(q_{N}\right) & = & \left\langle \mathbf{h}_{r},\mathbf{f}_{2}\left(\mathbf{h}_{r},q_{N}\right)\right\rangle \label{eq:MorseLemma}
\end{eqnarray}
}where $\mathbf{f}_{2}\left(\mathbf{h}_{r},q_{N}\right)$ is a unknown
but bounded function.

Suppose now that exist hybrid impact invariant functions $\mathbf{h}_{r}^{ref}\left(q_{N}\right)$,
$V\left(q_{N}\right)$ and $S\left(q_{N}\right)$ . Then if we choose

{\small
\begin{eqnarray}
\mathbf{v}_{r} & = & -K_{P}\mathbf{h}_{r}-K_{V}\mathcal{L}_{f}\mathbf{h}_{r}\label{eq:Chap2_ControlLaw_Vr}
\end{eqnarray}
}with $K_{P}$ and $K_{V}$ positive definite, this will ensure convergence
of $\mathbf{q}_{r}\rightarrow\mathbf{h}_{r}^{ref}\left(q_{N}\right)$.
The convergence $\mathbf{x}\rightarrow\mathcal{W}$ will be independent
of $q_{N}$, $\dot{q}_{N}$, $q_{d}$ or $\dot{q}_{d}$ and also $\mathcal{W}$
will be forward invariant.

The input $v_{N}$ can be freely used to drive the absolute orientation
of the robot. It is theoretically possible to track any reference
trajectory by the use of the output (\ref{eq:Chap2_OurOutput}) as
we found a linearizable part of dimension $2N-2$, the side effect
being that we do not have direct control over the speed accumulated
by the disc. But there is an interesting property of this robot, as
$q_{d}$ and $\dot{q}_{d}$ can be ignored for the dynamic model.
So the remaining coordinates will be $q_{r}$, $q_{N}$ , $\sigma_{N}$,
$\dot{q}_{r}$, $\dot{q}_{N}$. As $q_{N}$ is the only absolute coordinate,
$\mathbf{D}$ is also independent of $q_{N}$.

Now define\begin{subequations}\label{eq:NewCoordinates}{\small
\begin{eqnarray}
b & \triangleq & \sigma_{N}-S\left(q_{N}\right)\label{eq:NewCoordinates_b}\\
c & \triangleq & \dot{q}_{N}-V\left(q_{N}\right)-\cfrac{\gamma}{\beta_{0}}\cfrac{\partial S\left(q_{N}\right)}{\partial q_{N}}b\label{eq:NewCoordinates_c}
\end{eqnarray}
}\end{subequations}with $\gamma$ and $\beta_{0}$ positive constants.
It is possible to define the manifold {\small
\begin{eqnarray}
\mathcal{Z} & \triangleq & \left\{ \mathbf{x}\in\mathcal{W}|b=0,c=0\right\} \label{eq:manifold_Z}
\end{eqnarray}
}and it will be impact invariant. When the robot is performing the
nominal step, $\mathbf{x}\in\mathcal{Z}$ and $\mathcal{Z}\cap\mathcal{S}\in\mathbb{R}$.

As $q_{N}$ is monotonic at the continuous phase, it is possible to
integrate (\ref{eq:LinearizedSystem}) at $q_{N}$ and rewrite it
using the coordinates {\small
\begin{eqnarray}
\mathbf{y} & \triangleq & \left(\begin{array}{cccc}
\mathbf{h}_{r} & \dot{\mathbf{h}}_{r} & b & c\end{array}\right)^{T}\label{eq:NewVectorState}
\end{eqnarray}
}So, defining

{\small
\begin{eqnarray}
\xi\left(q_{N}\right) & \triangleq & \left(c\left(q_{N}\right)+V\left(q_{N}\right)+\nicefrac{\gamma}{\beta_{0}}\cfrac{\partial S}{\partial q_{N}}b\left(q_{N}\right)\right)^{-1}
\end{eqnarray}
}using the input $\mathbf{v}_{r}$ as defined in (\ref{eq:Chap2_ControlLaw_Vr})
and $v_{N}$ as

{\small
\begin{eqnarray}
v_{N} & = & \cfrac{\partial V}{\partial q_{N}}\dot{q}_{N}+\gamma\cfrac{\partial S}{\partial q_{N}}b-\beta_{0}\left(\dot{q}_{N}-V\right)+\cfrac{d}{dt}\left(\cfrac{\gamma}{\beta_{0}}\cfrac{\partial S}{\partial q_{N}}b\right)\label{eq:ControlLaw_vn}
\end{eqnarray}
} then the continuous part of the dynamic system, subject to the given
hypotheses, can be written as{\small
\begin{eqnarray}
\cfrac{d\mathbf{y}}{dq_{N}} & = & \bar{\mathbf{f}}\left(\mathbf{y},q_{N}\right)\label{eq:NewVectorFlowFunction}\\
\bar{\mathbf{f}}\left(\mathbf{y},q_{N}\right) & \triangleq & \xi\left(\begin{array}{c}
\dot{\mathbf{h}}_{r}\\
-K_{P}\mathbf{h}_{r}-K_{V}\dot{\mathbf{h}}_{r}\\
\left\langle \mathbf{h}_{r},\mathbf{f}_{2}\right\rangle -\tfrac{\partial S}{\partial q_{N}}c-\nicefrac{\gamma}{\beta_{0}}\left(\tfrac{\partial S}{\partial q_{N}}\right)^{2}b\\
-\beta_{0}c
\end{array}\right)
\end{eqnarray}
}Note also that by using the input (\ref{eq:ControlLaw_vn}), the
impact invariant manifold $\mathcal{Z}$ will be forward invariant
and thus, hybrid invariant.

In this new coordinates, the impact map will be\begin{subequations}\label{eq:NewSystemWithInput_ImpactMap}{\small
\begin{eqnarray}
\left(\begin{array}{c}
q_{d}^{+}\\
\mathbf{h}_{r}^{+}\\
q_{N}^{+}
\end{array}\right) & = & \Delta_{q}\left(\begin{array}{c}
q_{d}^{-}\\
\mathbf{h}_{r}^{-}+\mathbf{h}_{r}^{ref}\left(q_{N}^{-}\right)\\
q_{N}^{-}
\end{array}\right)-\left(\begin{array}{c}
0\\
\mathbf{h}_{r}^{ref}\left(q_{N}^{+}\right)\\
0
\end{array}\right)\label{eq:NewSystemWithInput_ImpactMap_A}\\
\left(\begin{array}{c}
b^{+}\\
\dot{\mathbf{h}}_{r}^{+}\\
c^{+}
\end{array}\right) & = & \Delta_{p}\left(h_{r}^{-},q_{N}^{-},q_{N}^{+}\right)\left(\begin{array}{c}
b^{-}\\
\dot{\mathbf{h}}_{r}^{-}\\
c^{-}
\end{array}\right)\label{eq:NewSystemWithInput_ImpactMap_B}
\end{eqnarray}
}\end{subequations}where

{\small
\begin{eqnarray}
\Delta_{p} & = & \Delta_{1}\left(q_{N}^{+}\right)\Delta_{\omega}\left(\mathbf{h}_{r}+\mathbf{h}_{r}^{ref}\left(q_{N}\right)\right)\Delta_{2}\left(q_{N}^{-}\right)\label{eq:NewImpactMapWithInput}\\
\Delta_{1} & = & \left(\begin{array}{ccc}
1 & 0 & 0\\
0 & 1 & -\cfrac{\partial\mathbf{h}_{r}^{ref}}{\partial q_{N}}\\
-\nicefrac{\gamma}{\beta_{0}}\cfrac{\partial S}{\partial q_{N}} & 0 & 1
\end{array}\right)\nonumber \\
\Delta_{2} & = & \left(\begin{array}{ccc}
1 & 0 & 0\\
\nicefrac{\gamma}{\beta_{0}}\cfrac{\partial\mathbf{h}_{r}^{ref}}{\partial q_{N}}\cfrac{\partial S}{\partial q_{N}} & 1 & \cfrac{\partial\mathbf{h}_{r}^{ref}}{\partial q_{N}}\\
\nicefrac{\gamma}{\beta_{0}}\cfrac{\partial S}{\partial q_{N}} & 0 & 1
\end{array}\right)\nonumber
\end{eqnarray}
}From (\ref{eq:NewSystemWithInput_ImpactMap_A}) and (\ref{eq:ImpactManifold})
it is possible to find the impact effect over $\mathbf{h}_{r}$ and
the actual integration limits $q_{N}^{+}$ and $q_{N}^{-}$, $q_{N}^{+}<q_{N}^{-}$.
From (\ref{eq:NewSystemWithInput_ImpactMap_B}) it is possible to
find how the remaining part of (\ref{eq:NewVectorState}) are affected
by the impact. The impact map with respect to the vector state (\ref{eq:NewVectorState})
will be{\small
\begin{eqnarray}
\mathbf{y}^{+} & = & \Delta_{\mathbf{y}}\left(\mathbf{y}^{-}\right)\label{eq:NewVectorImpactMap}
\end{eqnarray}
}and can be obtained by reshaping (\ref{eq:NewSystemWithInput_ImpactMap})
adequately. From (\ref{eq:NewVectorFlowFunction}) and (\ref{eq:NewVectorImpactMap}),
the complete hybrid system will be{\small
\begin{equation}
\begin{cases}
\cfrac{d\mathbf{y}}{dq_{N}}=\bar{\mathbf{f}}\left(\mathbf{y},q_{N}\right) & \mathbf{y}\notin\mathcal{S}\\
\mathbf{y}^{+}=\Delta_{y}\left(\mathbf{y}^{-}\right) & \mathbf{y}\in\mathcal{S}
\end{cases}\label{eq:NewHybridSystem}
\end{equation}
}

\subsection{Poincar\'{e} Map}

When walking, the dynamic system (\ref{eq:NewHybridSystem}) will
take a periodic evolution composed of a continuous phase, followed
by the impact map and another continuous phase. To evaluate the stability
of this periodic orbit, it will be important to define Poincar\'{e} return
map for the manifold $\mathcal{S}$ as{\small
\begin{eqnarray}
\mathcal{P} & : & \mathcal{S_{P}}\rightarrow\mathcal{S}
\end{eqnarray}
}where $\mathcal{S_{P}}\in\mathcal{S}$ is a neighborhood of $\mathbf{y}=\mathbf{0}$,
such that{\small
\begin{eqnarray}
\mathcal{P}\left(\mathbf{y}\right) & \triangleq & \varphi_{\mathbf{y}}\left(q_{N}^{+},q_{N}^{-},\Delta_{\mathbf{y}}\left(\mathbf{y}\right)\right)
\end{eqnarray}
}with $\varphi_{\mathbf{y}}\left(q_{N}^{+},q_{N}^{-},\mathbf{y}_{0}\right)$
being the integral curve of (\ref{eq:NewHybridSystem}) with initial
condition $\mathbf{y}_{0}$, starting at $q_{N}=q_{N}^{+}$ and ending
at $q_{N}=q_{N}^{-}$. So $\varphi_{\mathbf{y}}\left(q_{N}^{+},q_{N}^{+},\mathbf{y}_{0}\right)=\mathbf{y}_{0}$.

By the hybrid invariance of the manifold $\mathcal{Z}$, $\mathbf{y}_{0}=\mathbf{0}$
will be a fixed point of $\mathcal{P}${\small
\begin{eqnarray}
\mathcal{P}\left(\mathbf{0}\right) & = & \mathbf{0}
\end{eqnarray}
}As expected, there exist a periodic orbit of $\mathbf{y}\in\mathcal{Z}$
and, the interception of this periodic orbit with the impact surface
$\mathcal{S}$, that is $\mathcal{S}\cap\mathcal{Z}$, occours at
$\mathbf{y}=\mathbf{0}$.

\section{Stability around the nominal step}\label{sec:ControlStability}
\begin{thm}
Under the given robot hypotheses and the existence of $\mathbf{h}_{r}^{ref}\left(q_{N}\right)$,
$V\left(q_{N}\right)$ and $S\left(q_{N}\right)$ such that the manifolds
$\mathcal{W}$ and $\mathcal{Z}$ are hybrid invariants, then:

- The Poincar\'{e} return map $\mathcal{P}$ is regular in some
neighborhood $\mathcal{S_{P}}\subset\mathcal{S}$ around $\mathbf{y}=0$.

- The Poincar\'{e} return map $\mathcal{P}$ can be made locally exponentially contracting in
some neighborhood $\mathcal{S_{C}}\subset\mathcal{S_{P}}$ of $\mathbf{y}=0$,
by adjusting the parameters $K_{P}$, $K_{v}$, $\beta_{0}$ and $\gamma$.\end{thm}
\begin{cor}
If it is possible to find (\ref{eq:Chap2_LinearInput}) for every
possible value of $\mathbf{x}$, then the the system (\ref{eq:HybridSystem}),
except for the disc position, can be made exponentially stable around
some neighborhood $\mathcal{U}$ of $\mathcal{Z}$ by using the inputs
defined by (\ref{eq:Chap2_ControlLaw_Vr}), (\ref{eq:ControlLaw_vn})
and adjusting the parameters $K_{P}$, $K_{v}$, $\beta_{0}$ and
$\gamma$.\end{cor}
\begin{proof}
The Poincar\'{e} return map is well defined for $\mathbf{y}=\mathbf{0}$,
$\mathcal{Z}\cap\mathcal{S}$, and this is its fixed point. The manifold
$\mathcal{Z}$ is transversal to $\mathcal{S}$ and solutions of the
hybrid system (\ref{eq:NewHybridSystem}) are continuous. So for some
sufficiently small perturbation $\epsilon\in\mathcal{S}$ around $\mathbf{0}$
before the impact, the continuous part of the system will reach again
the impact surface. If $\mathcal{S_{P}}$ is defined to be the set
in which, in the event of a impact, at least one another impact will
follow, then it is possible to conclude that $\exists\mathcal{P}:\mathcal{S_{P}}\rightarrow\mathcal{S}$.

The impact map (\ref{eq:NewImpactMapWithInput}) can be linearized
at $\mathcal{Z}\cap\mathcal{S}$. This will lead to{\small
\begin{eqnarray}
\mathbf{y}^{+} & \approx & \bar{\Delta}_{\mathbf{y}}\mathbf{y}^{-}
\end{eqnarray}
}where $\bar{\Delta}_{\mathbf{y}}$ will be a matrix whose non constant
elements will increase at most with $\left(\nicefrac{\gamma}{\beta_{0}}\right)^{2}$.
Also, the integration limits $q_{N}^{+}$ and $q_{N}^{-}$, being
determined by the impact condition, will change if $\mathbf{h}_{r}\neq\mathbf{0}$.
They can be written as functions and when linearizated around $\mathbf{y}=\mathbf{0}$
will lead to{\small
\begin{eqnarray}
q_{N}^{+} & \approx & \bar{q}_{N}^{+}+\left\langle \mathbf{h}_{r}^{+},\bar{\mathbf{f}}_{+}\right\rangle \\
q_{N}^{-} & \approx & \bar{q}_{N}^{-}+\left\langle \mathbf{h}_{r}^{-},\bar{\mathbf{f}}_{-}\right\rangle
\end{eqnarray}
}where $\bar{\mathbf{f}}_{+}$ and $\bar{\mathbf{f}}_{-}$ are constants.

Note that $\left.\xi\left(q_{N}\right)\right|_{\mathcal{Z}}=\left(V\left(q_{N}\right)\right)^{-1}$,
so the differential equation (\ref{eq:NewVectorFlowFunction}), linearizated
around $\mathcal{Z}$, will be

{\small
\begin{eqnarray}
\cfrac{d\mathbf{y}}{dq_{N}} & \approx & \Xi\mathbf{y}\label{eq:AnalyticalLinearized}
\end{eqnarray}
}{\small \par}

with{\small
\begin{eqnarray*}
\Xi & = & V^{-1}\left(\begin{array}{cccc}
0 & 1 & 0 & 0\\
-K_{p} & -K_{v} & 0 & 0\\
f_{2} & 0 & -\nicefrac{\gamma}{\beta_{0}}\left(\cfrac{\partial S}{\partial q_{N}}\right)^{2} & -\cfrac{\partial S}{\partial q_{N}}\\
0 & 0 & 0 & -\beta_{0}
\end{array}\right)
\end{eqnarray*}
}{\small \par}

If the linearized system is exponentially stable, then the non linear
system will be stable around some neighborhood $\mathcal{U}$ of $\mathcal{Z}$.

We can find an explicit solution of (\ref{eq:AnalyticalLinearized})
for the k-nth step{\small
\begin{eqnarray}
\mathbf{y}_{\left(k\right)}^{-} & = & \Phi_{\mathbf{y}}\left(q_{N\left(k\right)}^{-},q_{N\left(k\right)}^{+}\right)\mathbf{y}_{\left(k\right)}^{+}\label{eq:ContinuousEvolutionMap-1}
\end{eqnarray}
}where{\footnotesize
\begin{eqnarray}
\Phi_{\mathbf{y}}\left(q_{N},q_{N}^{+}\right) & = & \left(\begin{array}{cccc}
g_{pp} & g_{pv} & 0 & 0\\
g_{vp} & g_{vv} & 0 & 0\\
g_{bp} & g_{bv} & g_{bb} & g_{bc}\\
0 & 0 & 0 & g_{cc}
\end{array}\right)
\end{eqnarray}
}{\footnotesize \par}

{\footnotesize
\begin{eqnarray}
g_{pp}\left(q_{N}\right) & = & \cfrac{\alpha_{1}e^{-\zeta\left(q_{N}\right)\alpha_{2}}-\alpha_{2}e^{-\zeta\left(q_{N}\right)\alpha_{1}}}{\alpha_{1}-\alpha_{2}}
\end{eqnarray}
\begin{eqnarray}
g_{vp}\left(q_{N}\right) & = & V^{-1}\alpha_{1}\alpha_{2}\cfrac{e^{-\zeta\left(q_{N}\right)\alpha_{1}}-e^{-\zeta\left(q_{N}\right)\alpha_{2}}}{\alpha_{1}-\alpha_{2}}
\end{eqnarray}
\begin{eqnarray}
g_{vv}\left(q_{N}\right) & = & V^{-1}\cfrac{\alpha_{1}e^{-\zeta\left(q_{N}\right)\alpha_{1}}-\alpha_{2}e^{-\zeta\left(q_{N}\right)\alpha_{2}}}{\alpha_{1}-\alpha_{2}}
\end{eqnarray}
\begin{eqnarray}
g_{cc}\left(q_{N}\right) & = & e^{-\beta_{0}\zeta\left(q_{N}\right)}
\end{eqnarray}
\begin{eqnarray}
g_{bb}\left(q_{N}\right) & = & e^{-\nicefrac{\gamma}{\beta_{0}}\psi\left(q_{N}\right)}
\end{eqnarray}
\begin{eqnarray}
g_{bc}\left(q_{N}\right) & = & e^{-\nicefrac{\gamma}{\beta_{0}}\psi\left(q_{N}\right)}\intop_{q_{N}^{+}}^{q_{N}}e^{\nicefrac{\gamma}{\beta_{0}}\psi\left(q_{N}\right)-\beta_{0}\zeta}\cfrac{\partial S}{\partial q_{N}}V^{-1}dq_{N}
\end{eqnarray}
}{\footnotesize \par}

{\footnotesize
\begin{eqnarray}
g_{bp}\left(q_{N}\right) & = & -\cfrac{e^{-\nicefrac{\gamma}{\beta_{0}}\psi\left(q_{N}\right)}}{\alpha_{1}-\alpha_{2}}\intop_{q_{N}^{+}}^{q_{N}}y_{bp}f_{2}\left(q_{N}\right)V^{-1}dq_{N}
\end{eqnarray}
}{\footnotesize \par}

{\footnotesize
\begin{eqnarray}
y_{bp} & = & \alpha_{2}e^{\nicefrac{\gamma}{\beta_{0}}\psi\left(\tau\right)-\alpha_{2}\zeta\left(\tau\right)}+\alpha_{1}e^{\nicefrac{\gamma}{\beta_{0}}\psi\left(\tau\right)-\alpha_{1}\zeta\left(\tau\right)}
\end{eqnarray}
\begin{eqnarray}
g_{bv}\left(q_{N}\right) & = & -\cfrac{e^{-\nicefrac{\gamma}{\beta_{0}}\psi\left(q_{N}\right)}}{\alpha_{1}-\alpha_{2}}\intop_{q_{N}^{+}}^{q_{N}}y_{bv}f_{2}\left(q_{N}\right)V^{-1}dq_{N}
\end{eqnarray}
\begin{eqnarray}
y_{bv} & = & e^{\nicefrac{\gamma}{\beta_{0}}\psi\left(\tau\right)-\alpha_{2}\zeta\left(\tau\right)}-e^{\nicefrac{\gamma}{\beta_{0}}\psi\left(\tau\right)-\alpha_{1}\zeta\left(\tau\right)}
\end{eqnarray}
\begin{eqnarray}
\alpha_{1,2} & = & K_{v}\pm\left(K_{v}^{2}-4K_{p}\right)^{\nicefrac{1}{2}}
\end{eqnarray}
\begin{eqnarray}
\zeta\left(q_{N}\right) & = & \intop_{q_{N}^{+}}^{q_{N}}V^{-1}\left(\tau\right)d\tau\geq0\forall q_{N}
\end{eqnarray}
\begin{eqnarray}
\psi\left(q_{N}\right) & = & \intop_{q_{N}^{+}}^{q_{N}}\left(\cfrac{\partial S\left(\tau\right)}{\partial q_{N}}\right)^{2}V^{-1}d\tau\geq0\forall q_{N}
\end{eqnarray}
}{\footnotesize \par}

It is possible to see that the parameters $K_{p}$, $K_{v}$, $\beta_{0}$
and $\gamma$ can be chosen large enough to have any initial condition
$\mathbf{y}_{\left(k\right)}^{+}$ at $q_{N}=q_{N\left(k\right)}^{+}$
attenuated as much as we want at the point $q_{N}=q_{N\left(k\right)}^{-}$,
end of the step. This will be true even if there are some small pertubation
on the integration limits due to $\mathbf{h}_{r}\neq\mathbf{0}$.

Choosing $\mathcal{S}$ as the Poincar\'{e} surface, it will be possible
to write its linearization as{\small
\begin{eqnarray}
\bar{\mathcal{P}}\left(\mathbf{y}_{0}\right) & \approx & \Phi_{\mathbf{y}}\left(q_{N}^{-},q_{N}^{+}\right)\bar{\Delta}_{\mathbf{y}}\mathbf{y}_{0}
\end{eqnarray}
}and this product  will be basically composed of terms like $\left(\nicefrac{\gamma}{\beta_{0}}\right)^{2}e^{-\nicefrac{\gamma}{\beta_{0}}}$
that, after some peak point, will exponentially became smaller as
$\nicefrac{\gamma}{\beta_{0}}$ increases. By adjusting the tunning
parameters, the linearized Poincar\'{e} map can be made as small as we
want and, as result, its eigenvalues can be allocated inside the unit
circle.
%
Even starting in a neighborhood of $\mathcal{W}$, where the integration
limits $q_{N}^{+}$ and $q_{N}^{-}$ will be different from the nominal
values, the stability is assured. In fact, appropriate values of $K_{P}$
and $K_{V}$ will lead, along each step, to the convergence of $q_{N\left(k\right)}^{+}$
and $q_{N\left(k\right)}^{-}$ to the nominal values $\bar{q}_{N}^{+}$
and $\bar{q}_{N}^{-}$.
\end{proof}
Under the conditions of the Corollary there exist a diffeomorphism
between (\ref{eq:NewHybridSystem}) and (\ref{eq:HybridSystem}),
except for the disc position. If the jacobian of the Poincar\'{e} return
map evaluated at $\mathcal{Z}\cap\mathcal{S}$ has its eigenvalues
inside the unit circle, then the system (\ref{eq:NewHybridSystem})
will be stable around some neighborhood $\mathcal{U}$ of $\mathcal{Z}$.
This can be achieved by choosing the parameters of the control law.
As conclusion, the system (\ref{eq:HybridSystem}) can be made stable
in the neighborhood $\mathcal{U}$, except for the disc position.

\section{Numerical simulations}\label{sec:Numerical-simulations}

By now some simulation results are presented for a model like the
one illustrated on picture \ref{fig:Chap2_ExampleBiped}. The corresponding author
can submit the numerical values and other details by email.

The proposed control law was simulated with parameters $\beta_{0}=4.5$
and $\gamma=0.35$, with the robot starting from rest at positions
$q_{N}=-10\degree$ and $q_{r}=-20\degree$. This means that the robot
state is outside $\mathcal{W}$. The simulated
behavior can be viewed on figure \ref{fig:Chap4_SpeedControl-2}.
\begin{figure}[h]
\begin{centering}
\includegraphics[width=1\columnwidth]{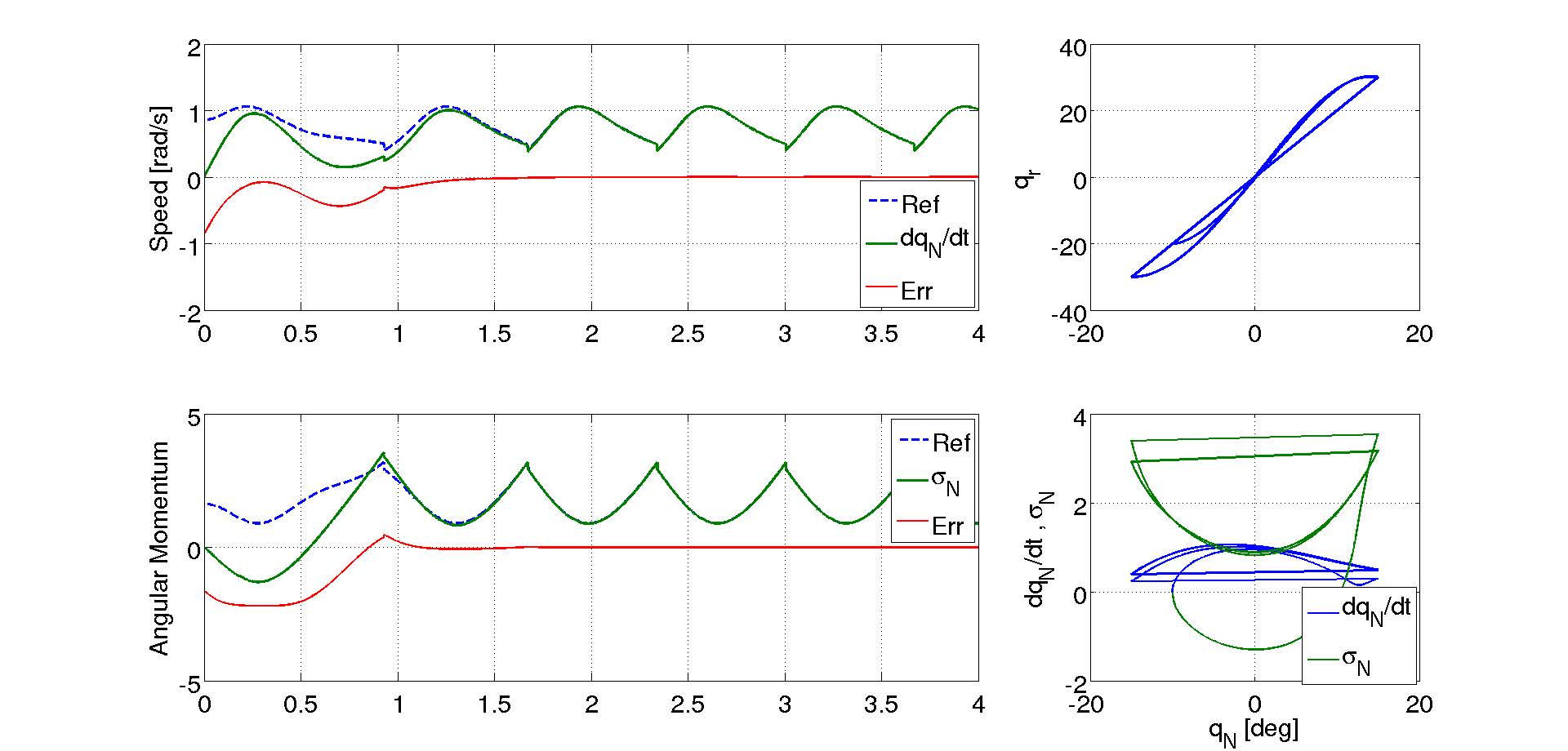}
\par\end{centering}

\caption{Complete control with exact model and starting outside $\mathcal{W}$
manifold.\label{fig:Chap4_SpeedControl-2}}
\end{figure}

Figure \ref{fig:Chap4_SpeedControl-3} show the results when some
parametric errors are included on the model. Anyway, the control remains
stable and can drive the robot close to the reference.
\begin{figure}[h]
\begin{centering}
\includegraphics[width=1\columnwidth]{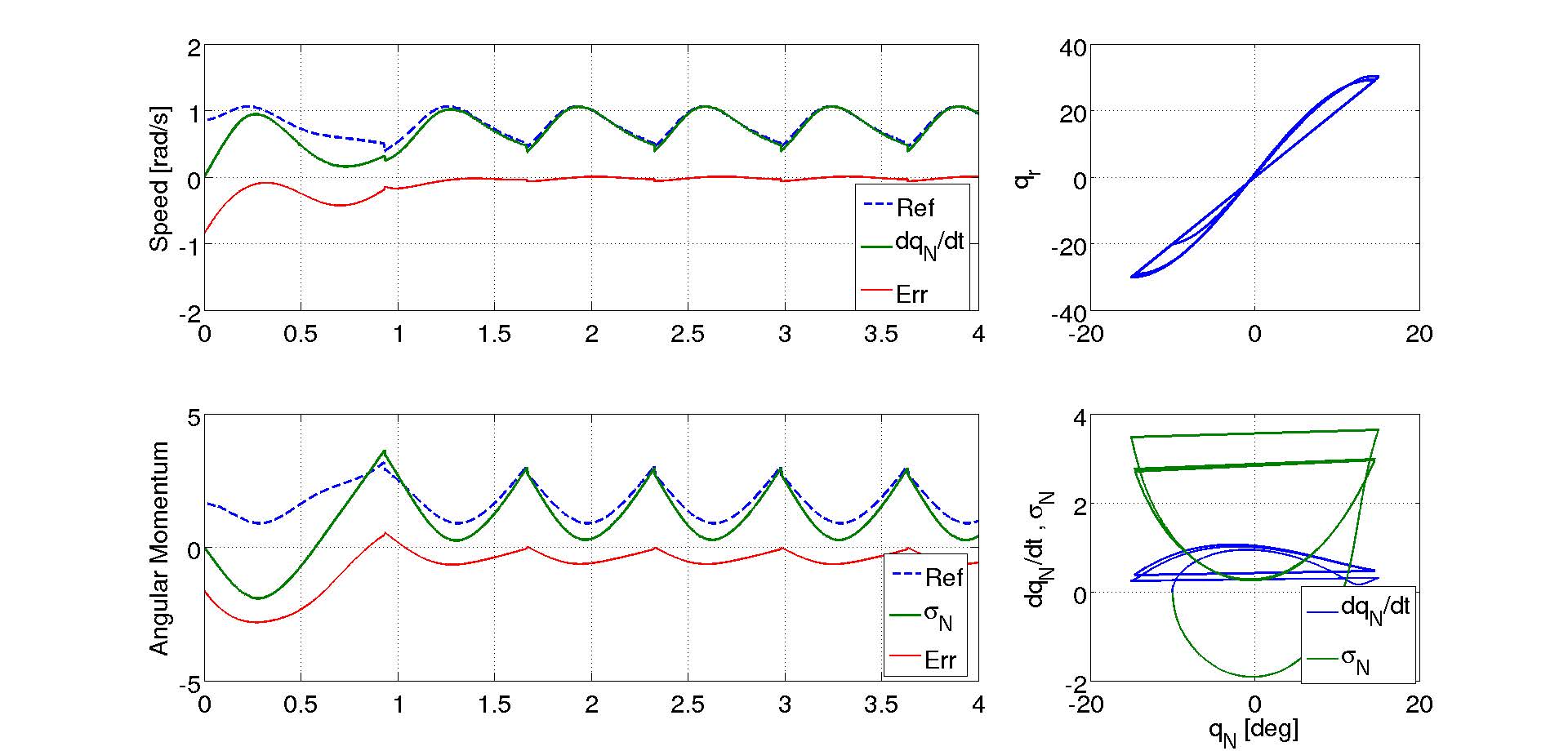}
\par\end{centering}

\caption{Complete control, model with parametric errors and starting outside
$\mathcal{W}$ manifold.\label{fig:Chap4_SpeedControl-3}}
\end{figure}

\section{Conclusions}

We have derived a control law for the presented class of bipedal robots,
so that the robot configuration will be tied to the absolute orientation.
Also the speed of this absolute orientation and the angular momentum
will follow predetermined references. As result, the robot will asymptotically
converge to a walking gait at the same time as the average disc speed
can be driven to zero (or any other desired value) by an appropriate
offset in the reference for the angular momentum.

The stability of the proposed control was proved around some neighborhood
of the nominal step and could be verified in numerical simulations.
The simulations shows that the domain of attraction is somewhat big
as the nominal step is reached even if the robot starts with no speed
and the results could also be validated for some parametric errors



\end{document}